\documentclass[12pt, a4paper]{amsart}
\usepackage{amsmath,amsfonts,amssymb, tikz, amsthm}
\usepackage[colorlinks=true,linkcolor=black,anchorcolor=black,citecolor=black,filecolor=black,menucolor=black,runcolor=black,urlcolor=blue]{hyperref}
\usepackage[utf8]{inputenc}
\usepackage[english]{babel}
\usepackage{xcolor}
\usepackage[tableposition=top]{caption} 
\usepackage{epigraph}
\usepackage{mathtools}

\newtheorem*{theorem*}{Theorem}
\newtheorem*{proposition*}{Proposition}
\newtheorem{theorem}{Theorem}
\newtheorem{lemma}{Lemma}
\newtheorem{proposition}{Proposition}

\newtheorem{conjecture}{Conjecture}

\theoremstyle{remark}
\newtheorem{example}{Example}
\newtheorem{remark}{Remark}

\theoremstyle{proof}

\numberwithin{equation}{section}
\numberwithin{assumption}{section}

\Alph{assumption}

\newcommand{\Z}{\mathbb{Z}}

\newcommand{\C}{\mathbb{C}}

\newcommand{\N}{\mathbb{N}}
\newcommand{\F}{\mathbb{F}}

\newcommand{\h}{\mathcal{H}}

\begin{document}
	
	\title{On variants of Chowla's conjecture}
	
	\author{Krishnarjun Krishnamoorthy}
	\email[Krishnarjun Krishnamoorthy]{krishnarjunmaths@outlook.com}
	\address{Beijing Institute of Mathemtical Sciences and Applications (BIMSA), No. 544, Hefangkou Village, Huaibei Town, Huairou District, Beijing.}

	\dedicatory{To my parents}
	

	\begin{abstract}
		We study the shifted convolution sums associated to completely multiplicative functions taking values in $\{\pm 1\}$ and give combinatorical proofs of two recent results in the direction of Chowla's conjecture. We also determine the corresponding "spectrum".
	\end{abstract}
	
	\maketitle
	
	\section{Introduction}\label{Section "Introduction"}
	
	\subsection{Completely multiplicative functions}
	A function $f : \N\to\C$ is called completely multiplicative if for every $m,n\in \N$, $f(mn) = f(m)f(n)$. It follows that completely multiplicative functions are completely described by their values at primes. Let $\lambda$ denote the Liouville function, defined to equal $-1$ at the primes and extended to natural numbers completely multiplicatively. The study of various averages related to the Liouville function is an important part of analytic number theory and is intimately tied with many outstanding conjectures such as the Riemann hypothesis etc. One such question regards the $d$-point correlations of the Liouville function, which is the focus of this paper. In particular we have the following conjecture of Chowla \cite{ChowlaConjecture}. 
	\begin{conjecture}[Chowla]\label{Conjecture "Chowla"}
		Let $\lambda$ denote the Liouville function and let $H = \{h_1, \ldots, h_d\}$ be a non-empty set of non-negative integers, then
		\begin{equation}\label{Equation "Chowla conjecture"}
			\lim_{x\to\infty}\frac{1}{x} \sum_{n\leqslant x} \lambda(n+h_1)\ldots\lambda(n+h_d) = 0.
		\end{equation}
	\end{conjecture}
	There are many versions of this conjecture available in literature, a few replacing the Liouville function with the M\"obius function (see \cite{RamareChowla} for more information and the inter-dependencies of various versions of Conjecture \ref{Conjecture "Chowla"}). In all the above variants, the conjecture encodes the expectation that the prime decomposition of an integer $n$ should be independent of that of $n+h$ for any fixed $h \geqslant 1$ (or other suitable polynomial shifts). Thus the values of Liouville function at these integers maybe treated as ``independent events'' and hence their average should vanish in the long run (that is they should not be correlated). Extending this philosophy to $d$-shifts gives Conjecture \ref{Conjecture "Chowla"}.

	\subsection{Main results}
	
	In this paper we give combinatorical proofs of two recent results regarding the shifted convolution averages of completely multiplicative functions taking values in $\{\pm 1\}$, particularly along the lines of Conjecture \ref{Conjecture "Chowla"}. We begin with some notation.  Let $P$ denote a subset of primes. Associated to $P$, we define the completely multiplicative function $\lambda_P$ as
	\begin{equation}\label{Equation "lambda_P definition"}
		\lambda_P(p) := \begin{cases}
			-1 &\mbox{ if }p\in P\\
			1&\mbox{ otherwise.}
		\end{cases}
	\end{equation}
	Equivalently, if $\Omega_P(n)$ denotes the completely additive function defined on the primes as
	\begin{equation}\label{Equation "Omega_P definition"}
		\Omega_P(p) := \begin{cases}
			1 &\mbox{ if }p\in P\\
			0&\mbox{ otherwise.}
		\end{cases}
	\end{equation}
	then $\lambda_P(n) = (-1)^{\Omega_P(n)}$. Every completely multiplicative function taking values in $\{\pm 1\}$ is of the form $\lambda_P$ for some $P$. In fact this association to a completely multiplicative function taking values in $\{\pm 1\}$, the set of primes where it takes the value $-1$ is a group isomorphism where the subsets of primes are considered as a group under symmetric difference and the completely multiplicative functions are considered under pointwise multiplication.
	
	Recall that a subset of natural numbers is called a \textit{small set} if its sum of reciprocals converges and called a \textit{large set} otherwise. For $H=\{h_1,\ldots,h_d\}$ we define
	\begin{equation}\label{Equation "Lambda definition"}
		\Lambda_P^H(n):= \lambda_P(n + h_1)\ldots \lambda_P(n+h_d).
	\end{equation}
	Define $\hat{H}:=\{h_i-h_j\ |\ h_i,h_j\in H, h_i\neq h_j\}$. Call a prime \textit{exceptional} if it divides an element of $\hat{H}$ and \textit{non-exceptional} otherwise. We now have the first main theorem of this paper.
	
	\begin{theorem}\label{Theorem "Main Theorem"} 
		Let $H = \{h_1, \ldots, h_d\}$ be a fixed subset of $\N$. Let $P$ be a small set of primes. Then, for every prime $p$, there exists constants $\eta_p^H$ (depending on $H$) such that
		\begin{equation}\label{Equation "Chowla"}
			\kappa_P^H:=\lim_{x\to\infty}\frac{1}{x} \sum_{n\leqslant x} \Lambda_P^H(n) =   \prod_{p\in P} \left(1 - 2\eta_p^H\right).
		\end{equation}
		Moreover, for every non-exceptional prime, $\eta_p^H = \frac{d}{p+1}$. 
	\end{theorem}
	The above theorem can be deduced from the results in \cite{Klurman}, but our proofs are different and comparitively elementary. What is interesting however is that the right hand side of \eqref{Equation "Chowla"} resembles an Euler product even when the corresponding arithmetic function namely $\Lambda_P^H(n)$ fails to be even multiplicative. Moreover, it is not even clear that the associated Dirichlet series (namely $\sum_{n=1}^{\infty} \Lambda_P^H(n) n^{-s}$) continues beyond the half plane $\Re(s) > 1$; a property much needed for the application of Tauberian arguments. The multiplicative nature of the right hand side of \eqref{Equation "Chowla"} in turn, conforms with the principle that the ``global'' density should be a product of the ``local'' densities.
	
	Theorem \ref{Theorem "Main Theorem"} still leaves open the question of large subsets of primes. In this case, we expect the limit on the left hand side of \eqref{Equation "Chowla"} to vanish, thus making Theorem \ref{Theorem "Main Theorem"} true in this case as well (if we interpret the infinite product on the right of \eqref{Equation "Chowla"} as zero, similar to the infinite product on the right hand side of \eqref{Equation "WW"} below). In particular, this would imply Conjecture \ref{Conjecture "Chowla"}. While such a generalization of Theorem \ref{Theorem "Main Theorem"} is out of reach, we are however able to prove the following.
	\begin{theorem}\label{Theorem "Second Main Theorem"}
		If there exists $\phi\neq H\subset \N$ such that $|\kappa_P^H|=1$, then $P=\phi$. More concretely, if $\lambda_P: \N\to \{\pm 1\}$ is a non constant completely multiplicative function, then for any $H = \{h_1,\ldots, h_d\}$,
		\begin{equation}\label{Equation "Liminf"}
			\liminf_{x\to\infty} \frac{1}{x} \sum_{n\leqslant x} \lambda_P(n+h_1)\ldots \lambda_P(n+h_d) < 1,
		\end{equation}
		and
		\begin{equation}\label{Equation "Limsup"}
			\limsup_{x\to\infty} \frac{1}{x} \sum_{n\leqslant x} \lambda_P(n+h_1)\ldots \lambda_P(n+h_d) > -1
		\end{equation}
	\end{theorem}
	Theorem \ref{Theorem "Second Main Theorem"} follows from a recent work of Teravainen \cite{JoniAMJ}, which is much more general. As mentioned before the purpose of the present article is to give an elementary combinatorical proof. 
	
	Given a family of multiplicative functions taking values inside the unit circle, their ``spectrum'' was studied in \cite{SoundSpectrum}. Analogously, we define the spectrum as $\Gamma_H$ to be the topological closure of the set of all values $\kappa_P^H$ as $P$ runs through all small subsets of primes for a fixed $H$. The precise evaluation of $\eta_p^H$ below allows us to describe the spectrum $\Gamma_H$ below in \S\ref{Section "Spectrum"}.
	
	\subsection{Brief overview of the proof}
	
	We give a brief overview of the proof of Theorems \ref{Theorem "Main Theorem"} and Theorem \ref{Theorem "Second Main Theorem"} leaving the details for the following sections. Define $	\mathcal{N}_P^H := \left\{ n\in\N\ |\ \Lambda_P^H(n) = -1 \right\}$	and set $\eta_P^H = \delta\left(\mathcal{N}_{P}^H\right)$ (we shall show later on that in the cases that interest us, this natural density exists).  The key idea in the proofs of Theorems \ref{Theorem "Main Theorem"} and \ref{Theorem "Second Main Theorem"} is the description of some structure in certain interesting families of subsets of $\N$.
	
	For Theorem \ref{Theorem "Main Theorem"}, the key ingredient is Lemma \ref{Lemma "N_p recursion"}, which describes how $\mathcal{N}_P^H$ behaves if we change $P$ and $H$. This allows us to prove Theorem \ref{Theorem "Main Theorem"} for finite sets of primes and to let the set of primes grow ``adding one prime at a time''. This leads to a limiting process which converges only if the original set $P$ is a small set. The relevant inequality is \eqref{Equation "Chain"}. As a consequence of the proof we describe how to calculate $\eta_P^H$ (see Example \ref{Example}). From \cite{Klurman}, $\eta_p^H$ is connected to the number of roots of the polynomial $(x+h_1)\ldots(x+h_d)$ modulo powers of $p$. We provide a slightly different description.
	
	The proof of Theorem \ref{Theorem "Second Main Theorem"} involves a combinatorical argument which is rather independent in itself. The idea is to show that the collection of counterexamples satisfy some additional structure and symmetries (see Lemma \ref{Lemma "Second main theorem 1"}). Using this we may reduce the the proof to the two-point correlation case, which is settled due to the work of Matomaki and Radziwill. The main arithmetic input, along with some handy lemmas, is Theorem \ref{Theorem "MatoRadz"} below (see also Remark \ref{Remark "2 set"}).
	
	\subsection{Known results}
	
	Before we proceed with the proof, we mention some important results available in literature. The only case where Conjecture \ref{Conjecture "Chowla"} is known is when $d=1$, in which case the statement is classically equivalent to the prime number theorem. Using the fact that $\lambda_P$ is completely multiplicative, we may derive (for $\Re(s) > 1$)
	\[
	\sum_{n=1}^{\infty} \frac{\lambda_P(n)}{n^s} = \zeta(s) \prod_{p\in P} \left(\frac{1-p^{-s}}{1+p^{-s}}\right)
	\]
	where $\zeta(s)$ is the Riemann zeta function. Combining the works of Wintner \cite{Wintner1} and Wirsing \cite{Wrising} leads to the following theorem (see also \cite[Theorem 2]{BorweinTransactions}).
	\begin{theorem}[Wintner-Wirsing]\label{Theorem "Wintner-Wirsing"}
		For any subset $P$ of the primes, we have
		\begin{equation}\label{Equation "WW"}
			\lim_{x\to\infty} \frac{1}{x} \sum_{n\leqslant x} \lambda_P(n) = \prod_{p\in P} \left(1-\frac{2}{p+1}\right).
		\end{equation}
	\end{theorem}
	An important aspect of the above theorem is the existence of the limit on the left hand side of \eqref{Equation "WW"}.
	
	Apart from the case of $d=1$, it was not even known whether the limit in \eqref{Equation "Chowla conjecture"} exists. Even if it did exist, it was unknown whether there is sufficient cancellations to ensure that the limit is better than the trivial limit (that is $\pm 1$). This second question was settled only recently, due to the important work of Matom\"aki and Radziwi\l\l\ \cite{MatoRadz}, where they studied the relation between ``short'' and ``long'' averages of arithmetic functions. We state, as a theorem, a corollary of the main result of \cite{MatoRadz}.
	
	\begin{theorem}[Matom\"aki - Radziwi\l\l]\label{Theorem "MatoRadz"}
		Let $h\geqslant 1$ be an integer, then there exists a positive constant $\delta(P,h)$ depending on $P$ and $h$ such that
		\[
		\frac{1}{x} \left|\sum_{n\leqslant x} \lambda_P(n)\lambda_P(n+h)\right| \leqslant 1-\delta(P,h)
		\]
		for all large enough $x > 1$ and for all $P$ non-empty. In other words, 
		\begin{equation}\label{Equation "MatoRadz"}
			\limsup_{x\to\infty} \frac{1}{x} \left|\sum_{n\leqslant x} \lambda_P(n)\lambda_P(n+h)\right| < 1
		\end{equation}
		for any $h\geqslant 1$, and $P\neq \phi$.
	\end{theorem}
	
	Most of the recent progress, particularly in the last decade or so regarding the Chowla's conjecture (or more generally the Hardy-Littlewood-Chowla conjectures) have been from the perspective of ergodic theory. Following \cite{MatoRadz}, there has been some progress towards Conjecture \ref{Conjecture "Chowla"} particularly focusing on averaged \cite{MatoRadzTao1, MatoRadzTao2} and logarithmic \cite{TaoForum, TaoTeravainenDuke} versions. For conditional results assuming the existence of Siegel zeros, see \cite{TaoTeravainenJLMS}. For more information\footnote{The author wishes to thank Stelios Sachpazis for directing him to these references.}, see also \cite{chinis2021siegelzerossarnaksconjecture,jaskari2024chowlaconjecturelandausiegelzeroes,TaoTeraBordeux, Lith1, Lith2,pilatte2023improvedboundstwopointlogarithmic, helfgott2021expansiondivisibilityparity}.

	\subsection*{Notation}
	
	\begin{enumerate}
		\item 	Let $\N = \{1,2,\ldots\}$ denote the natural numbers and $\N_0$ denote the non-negative integers $\N\cup\{0\}$.
		\item	The symbol $\delta$ will denote natural density (whenever it exists), $\delta^+$ and $\delta^-$ will denote the upper and lower natural densities respectively.
		\item 	We shall denote the empty set as $\phi$. The empty sum is defined to be zero and the empty product is defined to be one.
		\item 	Given any two subsets $A,B\subseteq \N_0$, we define their ``sum'' as $A+B := \{a+b\ |\ a\in A, b\in B\}$, and their ``product'' as $A\cdot B:= \{ab\ |\ a\in A, b\in B\}$. In particular, if $A=\{a\}$, we also write $a+B$ and $aB$ in place of $A+B$ and $A\cdot B$ respectively.
		\item 	The symbol $P$ will usually denote subsets of prime numbers unless otherwise stated. The symbol $H$ will usually refer to a finite subset of $\N_0$.
	\end{enumerate}
	
	\section{Groundwork}\label{Section "Groundwork"}
	
	Given a sequence $\{a_n\}\subset (-1,1)$, we define the infinite product $\prod_{n=1}^{\infty}(1-a_n)$ as $\lim_{T\to\infty} \prod_{n=1}^{T}(1-a_n)$. The product is considered to be convergent if and only if $\sum a_n$ is convergent, in which case the limit above exists. If $\sum a_n$ diverges to $\infty$, then the product is defined to be $0$. Given any two sets $A,B$,  their symmetric difference $A\triangle B$ is defined as $(A\setminus B)\cup (B\setminus A)$. In particular, $A\triangle B$ contains elements of $A\cup B$ belonging to exactly one of $A$ or $B$. Furthermore, by induction, the symmetric difference of a finite number of sets $\{A_i\}$ is precisely the collection of those elements belonging to exactly an odd number of the $A_i$'s. 	Suppose that $ P $ is a non-empty subset of primes, we let $ \mathcal{S}_P $ denote the set of $ P $ smooth integers. That is
	\[
	\mathcal{S}_P := \{n\in \N\ |\ \mbox{ all the prime factors of }n\mbox{ are in }P\}.
	\]
	It follows from our convention that $1\in \mathcal{S}_P$ for every $P$ and $\mathcal{S}_\phi = \{1\}$.
	\begin{lemma}\label{Lemma "Key lemma"}
		For any two subsets $P_1, P_2$ of the primes, and for any $n\in \N$,
		\[
		\lambda_{P_1\triangle P_2}(n) = \lambda_{P_1}(n)\lambda_{P_2}(n).
		\]
	\end{lemma}
	
	\begin{proof}
		Suppose first that $P_1$ and $P_2$ are disjoint. Factorize $ n $ as $ n = n_1n_2n_3 $ where $ n_1$ is the largest divisor of $n$ in  $\mathcal{S}_{P_1}$, $n_2$ that in $\mathcal{S}_{P_2} $ and $ n_3 $ defined as $ n/(n_1n_2) $. It follows that $ \lambda_{P_1}(n) = \lambda_{P_1} (n_1), \lambda_{P_2}(n) = \lambda_{P_2}(n_2) $ and $ \lambda_{P_1\cup P_2}(n) = \lambda_{P_1\cup P_2}(n_1n_2) $. Thus it suffices to show that 
		\begin{equation}\label{Equation "Key Lemma 1"}
			\lambda_{P_1}(n_1)\lambda_{P_2}(n_2) = \lambda_{P_1\cup P_2}(n_1n_2).
		\end{equation}
		But observe that $ \lambda_{P_1}(n_1) = \lambda(n_1) $ and likewise for the other two terms. Thus the lemma follows from the multiplicativity of the Liouville function.
				
		For the general case write $ n $ as the product $ n = n_1n_2n_3n_4 $ where $ n_1, n_2,n_3 $ are respectively the largest divisors of $n$ in $ \mathcal{S}_{P_1\setminus P_2}, \mathcal{S}_{P_1\cap P_2}, \mathcal{S}_{P_2\setminus P_1} $, and $ n_4 $ is defined as $ n/(n_1n_2n_3) $. Then on repeatedly applying \eqref{Equation "Key Lemma 1"}, we have
		\[
			\lambda_{P_1}(n)\lambda_{P_2}(n) = \lambda_{P_1\setminus P_2}(n) \lambda^2_{P_1\cap P_2}(n) \lambda_{P_2\setminus P_1}(n)
		= \lambda_{P_1\setminus P_2}(n) \lambda_{P_2\setminus P_1}(n)
		= \lambda_{P_1\triangle P_2}(n).
		\]
	\end{proof}
	
	From our conventions, $\mathcal{N}_P^\phi=\phi$ for any $P$.
	
	\begin{lemma}\label{Lemma "N_p recursion"}
		For any two subsets $ P_1, P_2 $ of the primes and two finite subsets $H_1,H_2$ of $\N_0$,
		\[
		\mathcal{N}_{P_1\triangle P_2}^{H_1\triangle H_2} = \mathcal{N}_{P_1}^{H_1} \triangle \mathcal{N}_{P_1}^{H_2} \triangle \mathcal{N}_{P_2}^{H_1} \triangle \mathcal{N}_{P_2}^{H_2}.
		\]
	\end{lemma}
	
	\begin{proof}
		The proof is given in two parts. First we show that $	\mathcal{N}_{P_1\triangle P_2}^{H} = \mathcal{N}_{P_1}^{H}\triangle \mathcal{N}_{P_2}^{H}$ for any $H$. If $H=\phi$, the claim follows from definition, so suppose not. Let $ m:= (n+h_1)\ldots(n+h_d) $, where we have chosen $H=\{h_1,\ldots,h_d\}$. From Lemma \ref{Lemma "Key lemma"} and complete multiplicativity,
		\[
		\Lambda_{P_1\triangle P_2}^H(n) = \lambda_{P_1\triangle P_2}^H(m) = \lambda_{P_1\setminus P_2}^H(m)\lambda_{P_2\setminus P_1}^H(m) = \Lambda_{P_1\setminus P_2}^H(n)\Lambda_{P_2\setminus P_1}^H(n).
		\]
		Thus $ \Lambda_{P_1\triangle P_2}^H(n) = -1 $ if and only if $ \Lambda_{P_1\setminus P_2}^H(n) \neq \Lambda_{P_2\setminus P_1}^H(n) $. But, we observe that $ \Lambda_{P_1}^H(n) = \Lambda_{P_1\setminus P_2}^H(n) \Lambda_{P_1\cap P_2}^H(n) $ and $ \Lambda_{P_2}^H(n) = \Lambda_{P_2\setminus P_1}^H(n)\Lambda_{P_1\cap P_2}^H(n) $. Thus, $ \Lambda_{P_1\setminus P_2}^H(n) \neq \Lambda_{P_2\setminus P_1}^H(n) $ if and only if $ \Lambda_{P_1}^H(n) \neq \Lambda_{P_2}^H(n) $. To conclude, $ \Lambda_{P_1\triangle P_2}^H(n)=-1 $ if and only if $ \Lambda_{P_1}^H(n) \Lambda_{P_2}^H(n) = -1 $, that is exactly one of them is $-1$. The latter condition is satisfied precisely on $ \mathcal{N}_{P_1}^{H} \triangle \mathcal{N}_{P_2}^{H} $.
		
		Now we show that $\mathcal{N}_P^{H_1\triangle H_2} = \mathcal{N}_P^{H_1}\triangle \mathcal{N}_P^{H_2}$. If either $H_1, H_2$ is the empty set, then the claim follows from definition, so we suppose not. Let $H_1=\{h_1,\ldots, h_{d_1}\}$ and $H_2 = \{k_1,\ldots, k_{d_2}\}$. Observe that $n\in \mathcal{N}_P^{H_1}$ if and only if $\Omega_P((n+h_1)\ldots(n+h_{d_1})) = \sum_{i=1}^{d_1}\Omega_P(n+h_i)$ is odd, and similarly for $H_2$. Let $H_1\triangle H_2 = \{l_1,\ldots,l_d\}$.
		
		If $n\in \mathcal{N}_P^{H_1\triangle H_2}$, then suppose without loss of generality that $\Omega_P(n+l_i)$ is odd for $1\leqslant i\leqslant e$ (say) (with $l_i\in H_1\triangle H_2$) and  even otherwise. In particular exactly one of $|\{l_1,\ldots,l_e\}\cap (H_1\setminus H_2)|$ or $|\{l_1,\ldots,l_e\}\cap (H_2\setminus H_1)|$ is odd. Suppose without loss of generality that the former holds. Let $\{r_1,\ldots, r_{e'}\}\subseteq H_1\cap H_2$ denote the complete set of integers such that $\Omega_P(n+r_i)$ is odd. If $e'$ is odd, then $n\in \mathcal{N}_P^{H_2}\setminus  \mathcal{N}_P^{H_1}$, and if $e'$ is even, then $n\in \mathcal{N}_P^{H_1}\setminus  \mathcal{N}_P^{H_2}$. In any case $n\in \mathcal{N}_P^{H_1}\triangle \mathcal{N}_P^{H_2}$. Therefore $\mathcal{N}_P^{H_1\triangle H_2}\subseteq \mathcal{N}_P^{H_1}\triangle \mathcal{N}_P^{H_2}$.
		
		Now we prove the reverse inclusion. Suppose without loss of generality that $n\in \mathcal{N}_P^{H_1}\setminus  \mathcal{N}_P^{H_2} $. Let us suppose (with notation as above) that $e'$ is even. The other case may be treated similarly. Choose $l_i$ as above. Then from our choices, it follows that $|\{l_1,\ldots,l_e\}\cap (H_1\setminus H_2)|$ is odd and $|\{l_1,\ldots,l_e\}\cap (H_2\setminus H_1)|$ is even. In particular, $e$ is odd and hence $n\in \mathcal{N}_P^{H_1\triangle H_2}$. This completes the proof.
	\end{proof}

	For $r,s\in \N$, define the arithmetic progression
	\[
	C(r,s):=\left\{n\in \N\ |\ n\equiv -r \mod s\right\}.
	\]
	It follows from the Chinese remainder theorem that if $ a,b $ are co-prime integers, then for any $ r,s $, there exists a unique $ t $ such that 
	\begin{equation}\label{Equation "CRT 0"}
		C(r,a)\bigcap C(s,b) = C(t,ab).
	\end{equation}
	Furthermore, 
	\begin{equation}\label{Equation "Density"}
		\delta(C(r,s)) = \frac{1}{s}
	\end{equation}
	for any $r,s\in \N$. Combining \eqref{Equation "CRT 0"} and \eqref{Equation "Density"}, we get
	\[
	\delta\left(C(r,a)\bigcap C(s,b)\right) = \delta\left(C(r,a)\right) \delta\left(C(s,b)\right)
	\]
	whenever $a,b$ are co-prime. Define $\mathcal{A}_P$ to be the collection all finite unions of the sets
	\[
	\{ \phi\}\cup\left\{C(r,n)\ |\ r\in \N,\ n\in \mathcal{S}_P\right\}.
	\]
	\begin{lemma}\label{Lemma "Properties of AP"}
		For any non-empty subset $P$ of the primes, the following statements are true.
		\begin{enumerate}
			\item	$\mathcal{A}_P$ is closed under intersection.
			\item 	$\mathcal{A}_P$ is closed under complements.
			\item 	Every element of $\mathcal{A}_P$ can be written as a finite disjoint union of sets of the form $\{C(r,n)\}$ for a fixed $n\in \mathcal{S}_P$.
			\item 	The natural density exists for every set $A\in \mathcal{A}_P$.
		\end{enumerate}
	\end{lemma}
	
	\begin{proof}
		\begin{enumerate}
			\item	It is enough to prove this assertion for sets of the form $C(r,n)$ for $n\in \mathcal{S}_P$. If $C(r_1,n_1)\cap C(r_2,n_2) = \phi$, we are done. Otherwise, $C(r_1, n_1)\cap C(r_2, n_2)$ is a finite union of sets of the form $C(r,n_3)$ where $n_3$ is least common multiple of $n_1$ and $n_2$. Therefore $n_3\in \mathcal{S}_P$ and we are done.
			\item 	From above, it is sufficient to show that $C^c(r,n)\in \mathcal{A}_P$. Clearly $C^c(r,n) = \cup_{s\neq r, 1\leqslant s\leqslant n} C(s,n)\in \mathcal{A}_P$.
			\item 	Suppose $n_1, \ldots, n_k\in \mathcal{S}_P$. Denote by $N$, the least common multiple of $\{n_1,\ldots, n_k\}$ and set $m_i = N/n_i$ for every $1\leqslant i \leqslant k$. We have,
			\[
			\bigcup_{i=1}^{k} C(r_i, n_i) = \bigcup_{i=1}^{k} \bigcup_{j=1}^{m_i} C(r_i + jn_i, N).
			\]
			The right hand side maybe written as a disjoint union by avoiding repetitions if there are any. 
			\item 	The natural density clearly exists for every set of the form $C(r,n)$ and the claim follows from above. This completes the proof.
		\end{enumerate}
	\end{proof}
		
	\begin{proposition}\label{Proposition "Multiplicative AP"}
		 Suppose that $P_1, P_2$ are non-empty mutually disjoint subsets of primes. Then $\delta(A\cap B) = \delta(A)\delta(B)$ for any $A\in \mathcal{A}_{P_1}$ and $B\in \mathcal{A}_{P_2}$.
	\end{proposition}
	
	\begin{proof}
		Suppose $A$ and $B$ are given as above. Then from Lemma \ref{Lemma "Properties of AP"}, there exists $N\in \mathcal{S}_{P_1}$ and $M\in \mathcal{S}_{P_2}$ such that 
		\begin{align*}
			A = \bigcup_{i=1}^{k_1} C(r_i, N) && \mbox{ and }&&B = \bigcup_{j=1}^{k_2} C(s_j, M),
		\end{align*}
		where each of the above is a \textit{disjoint union}. Therefore from the Chinese remainder theorem (see \eqref{Equation "CRT 0"}), there exists \textit{distinct} (modulo $NM$) $t_l$'s such that
		\[
		A\cap B = \bigcup_{l=1}^{k_1k_2} C(t_l, NM).
		\]
		From here, we may deduce that $\delta(A\cap B) = \frac{k_1k_2}{NM} = \delta(A)\delta(B)$.
	\end{proof}
	
	\section{Proof of Theorem \ref{Theorem "Main Theorem"}}\label{Section "Proof - Main"}
	
	 \subsection{Proof of Theorem \ref{Theorem "Main Theorem"} when $P$ is finite:}
	 
	 We simplify notation and write $\mathcal{N}_p^h$ for $\mathcal{N}_{\{p\}}^{\{h\}}$, $\Omega_{p}$ for $\Omega_{\{p\}}$ etc.	 
	 \begin{proposition}\label{Proposition "Approximation"}
	 	Given a finite subset $P$ of the primes and $H\subseteq \N_0$ and an integer $r$, there exists subsets $X_P^H(r), Y_P^H(r)$ in $\mathcal{A}_P$ such that 
	 	\[
	 	X_P^H(r)\subseteq \mathcal{N}_P^H\subseteq Y_P^H(r)
	 	\]
	 	and 
	 	\[
	 	\lim_{r\to\infty} \delta(X_P^H(r)) = \lim_{r\to\infty} \delta(Y_P^H(r)).
	 	\]
	 \end{proposition}
	
	\begin{proof}
		The proof is by double induction, first on $|H|$ and then on $|P|$. Let us first prove the statement for a fixed prime $p$ and when $H=\{h\}$. Then $\mathcal{N}_p^h$ contains precisely those integers $n$ such that an odd power of $p$ properly divides $n+h$. In particular,
		\begin{equation}\label{Equation "N_p^h union"}
			\mathcal{N}_p^h = \bigcup_{r=1}^{\infty} \left(C(h, p^{2r-1})\setminus C(h, p^{2r})\right).
		\end{equation}
		Since the above union is a disjoint union, we have the inequalities,
		\begin{equation}\label{Equation "X_p, Y_p 1"}
			\bigcup_{t=1}^{r} \left(C(h, p^{2t-1})\setminus C(h, p^{2t})\right) \subseteq \mathcal{N}_p^h \subseteq \bigcup_{t=1}^{r}\left( C(h, p^{2t-1})\setminus C(h, p^{2t})\right) \bigcup C(h, p^{2r-1})
		\end{equation}
		for any $r$. Call the sets on the left hand side and right hand side of \eqref{Equation "X_p, Y_p 1"} as $X_p^h(r)$ and $Y_p^h(r)$ respectively. We observe that $\delta(X_p^h(r))\leqslant \delta(Y_p^h(r))\leqslant \delta(X_p^h(r)) + p^{-2r+1}$. Letting $r\to\infty$, we see that $\delta(\mathcal{N}_p^h)$ exists and equals
		\begin{equation}\label{Equation "Temp 1"}
			\delta\left(\mathcal{N}_p^h\right) = \sum_{r=1}^{\infty} \left(\frac{1}{p^{2r-1}} - \frac{1}{p^{2r}}\right)= \frac{1}{p+1}.
		\end{equation}
		This completes the proof in this case.
		
		Suppose that $H= \{h_1,\ldots,h_d\}$ and consider $\mathcal{N}_p^H$. Let $\varphi : H\to \Z/p\Z$ denote the natural projection map. From Lemma \ref{Lemma "N_p recursion"},
		\begin{equation}\label{Equation "Temp"}
			\mathcal{N}_p^H =  \mathcal{N}_p^{\varphi^{-1}([0])}\triangle\ldots \triangle \mathcal{N}_p^{\varphi^{-1}([p-1])} = \mathcal{N}_p^{h_1}\triangle\ldots \triangle \mathcal{N}_p^{h_d}
		\end{equation}
		In particular, from \eqref{Equation "Temp 1"}, it follows that 
		\begin{equation}\label{Equation "eta_p evaluation"}
			\delta^+\left(\mathcal{N}_p^H\right)\leqslant \frac{d}{p+1}.
		\end{equation}
		But from \eqref{Equation "N_p^h union"}, we may easily deduce that $\mathcal{N}_p^{h_i}\cap \mathcal{N}_p^{h_j}\neq \phi$ only if $p| (h_i-h_j)$. Therefore we may rewrite the first equality in \eqref{Equation "Temp"} as the disjoint union
		\begin{equation}\label{Equation "Disjoint union"}
			\mathcal{N}_p^H = \bigcup_{i=0}^{p-1} \mathcal{N}_p^{\varphi^{-1}([i])}.
		\end{equation}
		If $\varphi$ is a non-constant function, then $\eta_p^H$ exists by induction on $|H|$ and we are done. In particular, we may suppose that $H\neq \{0,1\},\{1,2\}$ or $\{0,1,2\}$. 
		
		Suppose that $\varphi$ is a constant function, say $\varphi(H) = \{[i_1]\}$ for some $0\leqslant i_1\leqslant p-1$. If $n\not\equiv -i_1\mod p$, then $p\nmid (n+h)$ for any $h\in H$ and hence $\Lambda_p^H(n)=1$. Therefore, $n\in \mathcal{N}_p^H$ only if $n\equiv -i_1\mod p$. Moreover, in this situation,
		\[
		\Lambda_p^H(n) = \lambda_p(n+h_1)\ldots\lambda_p(n+h_d) = (-1)^d \lambda_p\left(\frac{n}{p} + \frac{h_1}{p}\right)\ldots \lambda_p\left(\frac{n}{p} + \frac{h_d}{p}\right).
		\]
		Let $H_1:=\frac{1}{p}(H-i_1)$. We observe that $\max\{H_1\} < \max\{H\}$. Rewriting the above,
		\begin{equation}\label{Equation "Calculation"}
			\mathcal{N}_p^H = \begin{cases}
				p \mathcal{N}_p^{H_1} - i_1&\mbox{if }2| d\\
				p \left(\mathcal{N}_p^{H_1}\right)^c - i_1&\mbox{otherwise.}
			\end{cases}
		\end{equation}
		In either case, it is enough to prove the proposition for $\mathcal{N}_p^{H_1}$. Now, if we consider the natural projection map $\varphi_1: H_1\to \Z/p\Z$, we may argue as above. If $\varphi_1$ is a non-constant map, we are done. Otherwise we obtain (if necessary, after translation and scaling as above) $H_2:= \frac{1}{p} (H_1-i_2)$, and the problem reduces to that of $\mathcal{N}_p^{H_2}$, with $\max\{H_2\} < \max\{H_1\}$. Proceeding forward this process ultimately terminates because $\max\{H_i\}$ is strictly decreasing. In particular, there is a stage $r$ where the projection map (say $\varphi_r : H_r\to \Z/p\Z$) is non-constant and we may then apply induction on $|H|$. This completes the proof when $|P|=1$. We note in passing that for large enough primes $p$, $\varphi$ is an injection and hence from \eqref{Equation "Temp 1"} and \eqref{Equation "Disjoint union"} we have $\eta_p^H = \frac{d}{p+1}.$
		
		Now fix $H$ and suppose that $P = \{p_1,p_2,\ldots, p_k\}$. We shall prove the proposition holds for $P$ supposing the same for $P':=\{p_1,p_2,\ldots, p_{k-1}\}$ and $\{p_k\}$. From Lemma \ref{Lemma "N_p recursion"}, 
		\[
		\mathcal{N}_P^H = \mathcal{N}_{P'}^H\triangle \mathcal{N}_{\{p_k\}}^H.
		\]
		For every $r$, choose $X_{P'}^H(r), Y_{P'}^H(r), X_{\{p_k\}}^H(r), Y_{\{p_k\}}^H(r)$ from induction. Then, we have
		\[
		\left(X_{P'}^H(r)\setminus Y_{\{p_k\}}^H(r)\right) \bigcup \left(X_{\{p_k\}}^H(r) \setminus Y_{P'}^H(r)\right) \subseteq \mathcal{N}_P^H \subseteq \left(Y_{P'}^H(r)\setminus X_{\{p_k\}}^H(r)\right) \bigcup \left(Y_{\{p_k\}}^H(r) \setminus X_{P'}^H(r)\right).
		\]
		From Lemma \ref{Lemma "Properties of AP"}, both the left hand side and the right hand side above belong to $\mathcal{A}_P$. In particular the natural densities exist for each of those sets. Considering (upper) natural densities (and observing that $X_{P'}^H(r)\subseteq Y_{P'}^H(r)$) we get,
		\begin{multline*}
			\delta\left(X_{P'}^H(r)\setminus Y_{\{p_k\}}^H(r)\right) +  \delta\left(X_{\{p_k\}}^H(r) \setminus Y_{P'}^H(r)\right)  \leqslant \delta^+(\mathcal{N}_P^H )\\
			\leqslant \delta\left(Y_{P'}^H(r)\setminus X_{\{p_k\}}^H(r)\right) + \delta \left(Y_{\{p_k\}}^H(r) \setminus X_{P'}^H(r)\right).
		\end{multline*}
		From Proposition \ref{Proposition "Multiplicative AP"}, we get
		\begin{multline*}
			\delta(X_{P'}(r))(1-\delta(Y_{\{p_k\}}(r))) + \delta(X_{\{p_k\}}(r))(1-\delta(Y_{P'}(r))) \leqslant \delta^+(\mathcal{N}_P) \\
			\leqslant \delta(Y_{P'}(r))(1-\delta(X_{\{p_k\}}(r))) + \delta(Y_{\{p_k\}}(r))(1-\delta(X_{P'}(r))).
		\end{multline*}
		Similarly, we may also consider lower natural densities. Letting $r\to\infty$, we prove $\delta(\mathcal{N}_P^H)$ exists and equals
		\begin{equation}\label{Equation "eta_P recursion"}
			\eta_{P}^H = \eta_{P'}^H(1-\eta_{p_k}^H) + \eta_{p_k}^H(1-\eta_{P'}^H).
		\end{equation}
	\end{proof}
	
	\begin{example}\label{Example}
		Given a particular choice of $H$ and $p$, it is possible to calculate $\eta_p^H$ using the above arguments, particularly using \eqref{Equation "Calculation"}. We briefly demonstrate this now for the choice of $p=2,3$ and $H=\{0,4,6\}$. Observe that $d=3$ is odd. On applying \eqref{Equation "Disjoint union"} and \eqref{Equation "Calculation"} repeatedly, we get 
		\[
		\eta_2^{\{0,4,6\}} = \frac{1}{2}\left(1 - \eta_2^{\{0,2,3\}}\right)	= \frac{1}{2}\left(1 - \left(\eta_2^{\{0,2\}} + \eta_2^{\{3\}}\right)\right)	= \frac{1}{2}\left(1 - \frac{(\eta_2^{\{0\}} + \eta_2^{\{1\}})}{2} - \eta_2^{\{3\}}\right) = \frac{1}{6}.
		\]
		where in the last step, we have used \eqref{Equation "Temp 1"}. Similarly,
		\[
		\eta_3^{\{0,4,6\}} = \eta_3^{\{0,6\}} + \eta_3^{\{4\}} = \frac{1}{3}\eta_3^{\{0,2\}} + \eta_3^{\{4\}} = \frac{1}{3}\cdot \frac{2}{4} + \frac{1}{4} = \frac{5}{12}.
		\]
	\end{example}
	
	To complete the proof of Theorem \ref{Theorem "Main Theorem"} when $P$ is finite, we observe that
	\begin{multline*}
		\kappa_{P}^{H}=\lim_{x\to\infty} \frac{1}{x} \sum_{n\leqslant x} \Lambda_{P}^H(n)
		=\\ \lim_{x\to\infty} \frac{1}{x}\left( \left|\left\{n\leqslant x\ |\ \Lambda_{P}^H(n)=1\right\}\right| -  \left|\left\{n\leqslant x\ |\ \Lambda_{P}^H(n)=-1\right\}\right| \right)
		= 1-2\eta_P^H.
	\end{multline*}
	Suppose that $P = \{p_1,p_2,\ldots, p_k\}$ and $P':=\{p_1,p_2,\ldots, p_{k-1}\}$ as above. From \eqref{Equation "eta_P recursion"}, 
	\begin{multline*}
		1 - 2\eta_P^H = 1 - 2\left(\eta_{P'}^H(1-\eta_{p_k}^H) + \eta_{p_k}^H(1-\eta_{P'}^H)\right)\\
		= 1- 2\eta_{P'}^H + 2\eta_{P'}^H\eta_{p_k}^H - 2 \eta_{p_k}^H + 2\eta_{p_k}^H\eta_{P'}^H
		= (1-2\eta_{P'}^H)(1-2\eta_{p_k}^H).
	\end{multline*}
	Now we may complete the proof via induction.
	
	 \subsection{Proof of Theorem \ref{Theorem "Main Theorem"} when $P$ is a small set:}	Suppose $H=\{h_1,\ldots,h_d\}$ is fixed as above and that $P = \{p_1,p_2,\ldots\}$ is an infinite (small) set and let $P_i = \{p_1, p_2, \ldots, p_i\}$. We show that $\eta_{P}^H$ exists and equals $\lim_{i\to\infty} \eta_{P_i}^H$. It will be convenient to set $S_P(x) = \frac{1}{x} \sum_{n\leqslant x} \Lambda_{P}^H(n) $. From Lemma \ref{Lemma "Key lemma"}, $ \Lambda_{P_i}^H(n)\neq \Lambda_{P}^H(n) $ if and only if $ \Lambda_{P\setminus P_{i}}^H(n) = -1 $. Hence, 
	\begin{equation}\label{Equation "Chain"}
		\limsup_{x\to\infty}\left|S_{P_i}(x) - S_P(x)\right| \leqslant 2\delta^+\left(\mathcal{N}_{P\setminus P_i}^H\right)\leqslant 2\sum_{p\in P\setminus P_i} \eta_p^H\leqslant 2d \sum_{p\in P\setminus P_i}\frac{1}{p+1},
	\end{equation}
	where the last inequality follows from \eqref{Equation "eta_p evaluation"}. As $P$ is a small set, the right hand side goes to zero as $i\to \infty$. This in particular shows $\eta_P^H$ exists because the left hand side of \eqref{Equation "Chain"} dominates
	\[
	\max\left\{\left|\limsup_{x\to\infty}S_{P_i}(x) - \liminf_{x\to\infty} S_P(x)\right|, \left|\limsup_{x\to\infty}S_{P}(x) - \liminf_{x\to\infty} S_{P_i}(x)\right|\right\}.
	\]
	Moreover, this also implies that $\eta_P^H = \lim_{i\to\infty} \eta_{P_i}^H$ thereby completing the proof.
	
	\begin{example}
		Interestingly if $H=\{1,2\}$ and $P$ is a small set containing $3$, the above calculations and proof give us that $\kappa_P^H=0$. More generally, for any given prime odd $p$, we may also choose $H= \{1,\ldots, \frac{p+1}{2}\}$ and any small set $P$ containing $p$ and get $\kappa_{P}^H=0$.
	\end{example}
			
	\section{Proof of Theorem \ref{Theorem "Second Main Theorem"}}\label{Section "Proof second main"}
	
	We denote the complement of $\mathcal{N}_P^H$ (in $\N$) as $\mathcal{P}_P^H$. Recall that $\kappa_P^{\phi}=1$.
	
	\begin{lemma}\label{Lemma "Second main theorem 1"}
		If there exists sets $P,H_1, H_2$ such that $|\kappa_P^{H_1}|=|\kappa_P^{H_2}|=1$, then $\kappa_P^{H_1\triangle H_2}=\kappa_P^{H_1}\kappa_P^{H_2}$.
	\end{lemma}
	
	\begin{proof}
	We have
		\[
		\kappa_{P}^{H_1} = \lim_{x\to\infty} \frac{1}{x} \sum_{n\leqslant x} \Lambda_{P}^{H_1\setminus H_2}(n) \Lambda_{P}^{H_1\cap H_2}(n)=\pm 1.
		\]
		In particular, this means that $\Lambda_P^{H_1\setminus H_2}(n) = \kappa_P^{H_1}\Lambda_P^{H_1\cap H_2}(n)$ on a set of natural density equal to $1$. Reversing the roles of $H_1,H_2$, we obtain that $\Lambda_P^{H_2\setminus H_1}(n) =  \kappa_P^{H_2}\Lambda_P^{H_1\cap H_2}(n)$ on a set of natural density equal to $1$, and in particular $\Lambda_P^{H_2\setminus H_1}(n)\Lambda_P^{H_1\setminus H_2}(n)=  \kappa_P^{H_1}\kappa_P^{H_2}$ on a set of natural density equal to $1$. Thus we have shown that
		\[
		\kappa_{P}^{H_1\triangle H_2} = \lim_{x\to\infty} \frac{1}{x} \sum_{n\leqslant x} \Lambda_{P}^{H_1\setminus H_2}(n) \Lambda_{P}^{H_2\setminus H_1}(n) =  \kappa_P^{H_1} \kappa_P^{H_2},
		\]
		completing the proof.
	\end{proof}
	
	We borrow the proof of the following lemma from combinatorics \cite{MarkWildon}.
	
	\begin{lemma}[Wildon]\label{Lemma "Second main theorem 2"}
		Suppose $\mathcal{H}$ is a non-empty collection of finite subsets of $\N_0$ which is closed under symmetric difference and translation, containing at least one non-empty set. Then $\h$ contains a set of two elements.
	\end{lemma}
	
	By translation, we mean that if $H\in \h$ and $a\in \Z$ is such that $H+a\subseteq \N_0$, then $H+a\in \h$.
	
	\begin{proof}
		For a given prime power $q$, let $\F_q$ denote the field of $q$ elements. Consider the map $\varphi : \h\to \F_2(t)$ given by $H\mapsto \sum_{h\in H} t^{h}$. By convention $\varphi(\phi)=0$. We observe that $\varphi(H_1\triangle H_2) = \varphi(H_1) + \varphi(H_2)$ and that $\varphi(H+a) = t^a \varphi(H)$. These two observations together tell us that $\varphi(\h)$ is an ideal of $\F_2(t)$. Suppose this ideal is generated by $f(t)$. Because $\h$ is closed under translation as above, $f(0)\neq 0$ (equivalently $t\nmid f(t)$). Let $\F_{2^r}$ be the splitting field of $f$. Allowing for multiple roots if needed, we see that there exists an integer $m$ such that $f(t) | (t^{2^{r}-1}+1)^m$, that is, there exists a polynomial $q(t)$ such that $q(t)f(t) = (t^{2^{r}-1}+1)^m$. If $n$ is so that $2^n\geqslant m$, then
		\[
		f(t) g(t)(t^{2^{r}-1}+1)^{2^n-m} = (t^{2^{r}-1}+1)^{2^n} = t^{(2^{r}-1)2^n}+1.
		\]
		But the left hand side is an element of $\varphi(\h)$ and hence $\{1, (2^{r}-1)2^n+1\}\in \h$.
	\end{proof}
	
	\begin{proof}[Proof of Theorem \ref{Theorem "Second Main Theorem"}]
		Suppose that $P$ is given such that $|\kappa_{P}^{H_0}|=1$ for some (non-empty) $H_0$. Let $\h_P$ denote the collection of all finite subsets $H$ of $\N_0$ such that $|\kappa_P^H|=1$. Clearly $\h_P$ is a non-empty collection and is closed under translation.  From Lemma \ref{Lemma "Second main theorem 1"}, $\h_P$ is closed under symmetric difference. Then by Lemma \ref{Lemma "Second main theorem 2"}, $\h_P$ contains a set of order two. This however contradicts Theorem \ref{Theorem "MatoRadz"} unless $P=\phi$. This completes the proof.
	\end{proof}
	
	\begin{remark}\label{Remark "2 set"}
		A family $\h$ satisfying the conditions of Lemma \ref{Lemma "Second main theorem 1"} need not contain a singleton element. Thus, the above argument fails without the knowledge of Theorem \ref{Theorem "MatoRadz"}. In particular, Theorem \ref{Theorem "Wintner-Wirsing"} is not sufficient to deduce Theorem \ref{Theorem "Second Main Theorem"}.
	\end{remark}
	
	\begin{remark}
		It is an desirable to prove an analogue of \eqref{Equation "MatoRadz"} in place of \eqref{Equation "Liminf"} and \eqref{Equation "Limsup"}. However, the current proof does not seem to yield that strengthening. The challenge is to improve upon Lemma \ref{Lemma "Second main theorem 1"}. More precisely, if two subsets of natural numbers have natural density equal to $1$, then so does their intersection and this natural density is once again equal to $1$. However, this is no longer true for upper or lower natural densities.
	\end{remark}

	\section{The spectrum of $\kappa_P^H$}\label{Section "Spectrum"}
	
	Let $H\neq \phi$ be given. Define $\alpha_H:=\inf_{p}\{(1-2\eta_p^H)\}$ where the infimum runs over all the primes. Furthermore, we remark that there exists a prime $p$ such that $\alpha_H = 1-2\eta_p^H$.
	
	\begin{proposition}\label{Proposition "Spectrum"}
		With notation as above, $\Gamma_H = [\alpha_H,1] \cup [0,1]$.
	\end{proposition}
	
	Before we move on to the proof of Proposition \ref{Proposition "Spectrum"}, we require the following lemma. We give the proof for the sake of completeness. 
	
	\begin{lemma}\label{Lemma "Spectrum"}
		For any given $H$, $\alpha\in(0,1)$ and a parameter $X$, there exists a small set of primes $P$, with each prime larger than $X$ such that $\kappa_P^H = \alpha$.
	\end{lemma}
	
	\begin{proof}
		We shall choose $X$ large enough without loss of generality and suppose that every prime $p\geqslant X$ is a non-exceptional prime. Choose a finite set $P_1$ such that $\frac{1+\alpha}{2} > \kappa_{P_1}^H > \alpha$. Such a choice is possible because the sum of reciprocals of primes is divergent. For simplicity, suppose $\alpha_1 = \alpha/\kappa_{P_1}^H$. Choose $P_2$ disjoint from $P_1$ such that $\frac{1 + \alpha_1}{2} > \kappa_{P_2}^H > \alpha_1$. Define $\alpha_2 : = \alpha_1/\kappa_{P_2}^H$. In the next step, we choose $P_3$ disjoint from $P_1\cup P_2$ and so on. Constructing $\alpha_i$'s recursively as above, give for each $i$
		\[
		\frac{1 + \alpha_i}{2} > \kappa_{P{i+1}}^H > \alpha_i.
		\]
		Multiplying out by $\prod_{j=1}^{i} \kappa_{P_j}^H$ gives us
		\begin{equation}\label{Equation "inequality"}
			\frac{\prod_{j=1}^{i} \kappa_{P_j}^H + \alpha}{2} > \prod_{j=1}^{i+1} \kappa_{P_j}^H > \alpha.
		\end{equation}
		From the monotone convergence theorem, $\prod_{j=1}^{i+1} \kappa_{P_j}^H $ is convergent. But then \eqref{Equation "inequality"} forces the limit to be $\alpha$. Now we choose $P$ to be $\cup_{i=1}^{\infty} P_i$. The fact that $\kappa_P^H=\alpha$ follows from Theorem \ref{Theorem "Main Theorem"} and the fact $P$ is a small set follows from the fact that $\alpha > 0$. 
	\end{proof}
	
	\begin{proof}[Proof of Proposition \ref{Proposition "Spectrum"}] Choose $X$ so that for primes $p > X$, $\eta_p^H > \alpha_H$. If $\alpha_H \geqslant 0$, then the proposition is clear from Lemma \ref{Lemma "Spectrum"}. Otherwise choose $\beta\in [\alpha,1]$, we shall show that there exists a $P$ such that $\kappa_P^H = \beta$. From Lemma \ref{Lemma "Spectrum"}, we may suppose that $\beta\in (\alpha_H,0)$. Choose $\alpha : = \beta/\alpha_H$, and choose $P$ a small set of non-exceptional primes, each larger than $X$, such that $\kappa_P^H = 1-2\alpha$. Let $p$ be such that $\alpha_H = 1-2\eta_p^H$. Then from Theorem \ref{Theorem "Main Theorem"}, the required small set of primes is $P\cup \{p\}$. Thus $[\alpha_H,1]\cup [0,1] \subseteq \Gamma_H$. 
		
	Suppose $\alpha_H > 0$. Suppose $x\in \Gamma_H$, it suffices to show that $x \geqslant 0$. This is clear from Theorem \ref{Theorem "Main Theorem"} since in this case $(1-2\eta_p^H) > 0$ for every prime $p$, so that $\Gamma_H \subset [0,1]$.
	
	Suppose $\alpha_H < 0$. For every small set $P$, it then suffices to show that $\kappa_P^H \geqslant \alpha_H$. Now observe that $|\kappa_P^H| \leqslant 1$ for every subset $P$ of the primes. Suppose $1-2\eta_p^H = \alpha_H$ for some fixed prime $p$. For any small set $P$, we have that $\kappa_P^H \geqslant \kappa_{P\setminus \{p\}}^H \alpha_H \geqslant \alpha_H$.
	\end{proof}

	\bibliographystyle{plain}
	
	\bibliography{Bibliography}

\end{document}